\documentclass[11pt]{amsart}
\usepackage{amssymb, amsmath, amsthm}
\usepackage{amsrefs}
\usepackage[T1]{fontenc}
\usepackage[utf8]{inputenc}
\usepackage{lmodern}
\usepackage{graphicx}
\usepackage[final]{hyperref}
\usepackage[a4paper, centering]{geometry}
\usepackage{color}
\usepackage{graphicx}

\newtheorem{theorem}{Theorem}[section]

\newtheorem{lemma}[theorem]{Lemma}
\newtheorem{corollary}[theorem]{Corollary}
\theoremstyle{definition}

\theoremstyle{remark}
\newtheorem{remark}[theorem]{Remark}
\newtheorem{example}[theorem]{Example}
\def\R{\mathbb{R}}

\def\haus{\mathcal{H}^{N-1}}

\def\gtilde{\widetilde{g}}
\def\htilde{\widetilde{h}}
\def\Ftilde{\widetilde{F}}
\def\Fhat{\widehat{F}}
\def\Fhatrad{\Fhat_{\textrm{rad}}}

\def\one{\mathbb{I}}
\def\Frad{F_{\textrm{rad}}}

\newcommand{\rad}[1][u]{\overline{#1}}
\newcommand{\ACloc}[1][{]0,R]}]{AC_{\textrm{loc}}(#1)}
\newcommand{\W}[1][1]{\mathcal{W}_{\mu}^{#1}}

\newcommand{\Wbase}[1][1]{\mathcal{W}_{\textrm{rad}}^{#1}}
\newcommand{\Wbasestar}[1][1]{\mathcal{W}_{\textrm{rad}}^{#1,\ast}}

\newcommand{\Flambda}[1][\lambda]{F_{\varphi, #1}}
\newcommand{\ulambda}[1][\lambda]{\rad_{\varphi, #1}}
\newcommand{\glambda}[1][\lambda]{g_{\varphi, #1}}
\newcommand{\plambda}[1][\lambda]{p_{\varphi, #1}}

\DeclareMathOperator{\conv}{conv}

\DeclareMathOperator{\argmin}{argmin}

\DeclareMathOperator{\interior}{int}
\DeclareMathOperator{\Dom}{Dom}

\begin{document}

\title[Non-coercive radially symmetric variational problems]%
{Non-coercive radially symmetric \\ variational problems: \\
Existence, symmetry and convexity of minimizers}%
\author[G.~Crasta, A.~Malusa]{Graziano Crasta,  Annalisa Malusa}
\address[Graziano Crasta]{Dipartimento di Matematica ``G.\ Castelnuovo'', 
Sapienza Universit\`a di Roma\\
P.le A.\ Moro 5 -- 00185 Roma (Italy)}
\email{crasta@mat.uniroma1.it}

\address[Annalisa Malusa]{Dipartimento di Matematica ``G.\ Castelnuovo'', 
Sapienza Universit\`a di Roma\\
P.le A.\ Moro 5 -- 00185 Roma (Italy)}
\email{malusa@mat.uniroma1.it}



\date{April 23, 2019}

\begin{abstract}
We prove existence of radially symmetric solutions and validity of 
Euler--Lagrange necessary conditions for a class of variational problems
such that neither direct 
methods nor indirect 
methods of Calculus of Variations apply. 
We obtain existence and qualitative properties of the solutions
by means of \textsl{ad-hoc} superlinear perturbations of the functional having the same
minimizers of the original one. 
\end{abstract}

\maketitle

\section{Introduction}

This paper is concerned with the variational problem
\[
\min_{u\in W^{1,1}_0(B_R)}\int_{B_R} \left[
g(|x|, |\nabla u|) + h(|x|,u)
\right]\,dx\,,
\]
where $B_R\subseteq \R^N$ is the open ball centered at the origin and with 
radius $R>0$.

Under the sole assumptions of increasing monotonicity of the Lagrangian with 
respect to the gradient variable one can prove, by means of a 
symmetrization procedure proposed in \cite{Kro}, that the problem 
admits a one--dimensional reduction, obtained by evaluating the functional only 
on the set of radially symmetric functions (see Section \ref{s:radial}).

This reduction step leads to consider the minimum problem
\[
\min_{u\in\Wbase}
\int_0^R r^{N-1} [g(r, |u'(r)|) 
+ h(r,u(r))]\, dr
\]
on the space
\[
\Wbase := \left\{
u \in AC_{\textrm{loc}}(]0,R]):\
u(R) = 0,\
r^{N-1}\, |u'(r)|\in L^1(]0,R[)
\right\}\,.
\]
The qualitative features of the Lagrangian are that $g(r,\cdot)$ is convex (in fact 
this assumption can be dropped in the autonomous case, see Corollary 
\ref{c:nonconvex}) and with, at least, linear growth,
while  $h(r,t)$ is Lipschitz continuous in the $t$ variable. 
These assumptions do not assure that every minimizing sequence of the functional is 
precompact in $L^1$, and hence the direct methods of Calculus of Variations fails.

For this reason indirect methods, based on the solvability of the associated 
Euler--Lagrange equations, have often been adopted in the literature (see 
\cite{Cel04,CeTrZa,Clar93,ClarLo,C1,C2,C3,C4,CM1,CM2,CM3,Rock71}). 
Specifically, if the Lagrangian is convex with respect to both variables
$u$ and $|u'|$, then any solution of the Euler--Lagrange conditions
provides a minimizer, and vice-versa.

\smallskip
The main feature of the present work is that we do not require
convexity of the Lagrangian in the $u$ variable,
so that the above mentioned indirect methods cannot be implemented,
and a brand-new approach is needed.

Our starting points are an existence result and the validity of the 
Euler--Lagrange necessary conditions under the additional requirement that
$g(r,\cdot)$ has superlinear growth. 
These properties can be easily 
obtained applying well-known results 
(see Step~1 of the proof of Theorem~\ref{t:noncoercive}). 
Exploiting the necessary conditions,  
we obtain explicit \textsl{a-priori} estimates
on the derivative of minimizers of superlinear functionals,
that depend on the Lipschitz constant of $h(r,\cdot)$.

When $g(r,\cdot)$ satisfies only a linear growth condition,
say $g(r,s) \geq M\, s - C$ for some positive constants $M$ and $C$, 
and the Lipschitz constant of $h(r, \cdot)$ is not too large
compared with $M$
(see the compatibility relation (hgr) between $g$ and $h$ in the
statement of Theorem~\ref{t:noncoercive}),
then we proceed as follows.
As a first step,
we construct an \textsl{ad-hoc} superlinear perturbation of the slow growth functional,
for which we have a Lipschitz minimizer satisfying some \textsl{a-priori} estimates.
Then, relying on these estimates, we show that this function
is in fact a minimizer of the original slow-growth problem.

In some sense, our technique is reminiscent of the semiclassical approach,
based on the construction of barrier functions, for the minimization  
of functionals of the type $\int_{\Omega} L(\nabla u)\, dx$
on functions $u\in W^{1,1}(\Omega)$ satisfying some prescribed
boundary condition (see, e.g., \cite[Chapter~1]{GiDirect}).

As an application of our results, in Section \ref{s:convex} we
prove existence of convex Lipschitz continuous minimizers for 
variational problems with a constraint on the gradient. 
For related convexity results, obtained by
means of convex rearrangements, see \cite{Greco2012,Carlson}.  

Finally, we believe that our techniques
can be successfully implemented also for minimization problems related to
slow-growth integral functionals 
$\int_\Omega [ g(|\nabla u|) + h(u)]\, dx$
in a space of functions depending only on the distance from the boundary
of $\Omega$
(see, e.g., \cite{C6,C7,C8,CFG1,CFG2,CFG3,CFG4,CG1,CG2,CM4,CM5,CM9}).

\section{Notation and preliminaries}
\label{s:prelim}

In what follows $|\cdot|$ will denote the Euclidean norm in $\R^N$, $N\geq 1$,
and $B_R\subset \R^N$ is the open ball centered at the origin and with radius 
$R>0$.

We shall denote by $\overline{A}$ and $\interior{A}$ respectively the closure
and the interior of a set $A$, and by $\Dom{\varphi}$ the essential domain of
an extended real--valued function $\varphi\colon A \to ]{-}\infty,+\infty]$,
i.e. $\Dom{\varphi}=\{x\in A\colon\ \varphi(x)<+\infty\}$. 
We shall always 
consider \textsl{proper functions}, that is $\Dom{\varphi}\neq \emptyset$.

\medskip

Given a locally Lipschitz function $\varphi\colon A\subset\R \to \R$,
for every $x\in A$ we denote by $\partial \varphi(x)$ its
\textsl{generalized gradient} at $x$ in the sense
of Clarke (see \cite[Chapter~2]{Clar}).
We recall that, if $x$ is an interior point of $A$,
then $\partial \varphi(x)$ is a non-empty, convex, compact set
(see \cite[Proposition~2.1.2(a)]{Clar}).
Moreover, if $D\subset A$ denotes the set of points where $\varphi$
is differentiable, then
\[
\partial \varphi (x) =
\conv\left\{
\varphi'(x_j):\
(x_j) \subset D, \
x_j \to x
\right\}
\]
(see \cite[Theorem~2.5.1]{Clar}).
Hence, if $\varphi\colon\R\to\R$ is a monotone non-decreasing
$K$-Lipschitz function, then 
$\emptyset\neq\partial\varphi(x) \subseteq [0,K]$
for every $x\in\R$.

For notational convenience, if $\varphi$ also depends on an additional variable
$r\in\R$, we denote by $\partial\varphi(r,x)$ the generalized gradient of 
the function $x \mapsto \varphi(r,x)$.

\medskip

If $\varphi\colon\R\to ]{-}\infty, +\infty]$ is a lower semicontinuous convex function,
the generalized gradient $\partial\varphi(x)$ coincides with the subgradient 
(in the sense of convex analysis) at every point $x\in\interior\Dom\varphi$, 
and hence $\partial\varphi(x) = 
[\varphi'_-(x), \varphi'_+(x)]$, where $\varphi'_-(x)$ and $\varphi'_+(x)$
are the left and right derivative of $\varphi$ at $x$ 
(see \cite[Proposition~2.2.7]{Clar}). We shall often use the following 
implication, due to the monotonicity of 
the subgradient:
\[
p\in \partial \varphi(x),\ q\in \partial \varphi(y),\ \text{and}\ p<q
\ \Longrightarrow \ x\leq y.
\]

\medskip

If $\varphi\colon\R\to ]{-}\infty, +\infty]$, we denote by 
$\varphi^*$ its Fenchel--Legendre transform, or polar function
(see \cite[Section~I.4]{EkTem}).
With some abuse of notation, if
$\varphi\colon[0,+\infty[\to ]{-}\infty, +\infty]$, we use
$\varphi^*$ to denote the Fenchel--Legendre transform of
the even function $\R\ni z \mapsto \varphi(|z|)$, so that
\begin{equation*}
\varphi^*(p)=\sup_{x\in \R}\{p\, x-\varphi(|x|)\}.
\end{equation*}
We remark that, in this case, $\varphi^*$ is a lower semicontinuous convex even function. 

If $\varphi$ is a lower semicontinuous convex function, its subgradient and the subgradient 
of the polar function are related in the following way:
\[
p\in \partial \varphi(x) \Longleftrightarrow x \in \partial \varphi^*(p).
\]  
(see \cite[Corollary~I.5.2]{EkTem}).

\medskip

We say that $f\colon [0,R]\times \R \times [0,+\infty[\to ]{-}\infty, +\infty]$ 
is a
\textsl{normal integrand} if $f(r, \cdot, \cdot)$ is lower semicontinuous
for almost every (a.e.) $r\in [0,R]$, and there exists a Borel
function $\widehat{f}\colon [0,R]\times[0,+\infty[\to ]{-}\infty, +\infty]$
such that $\widehat{f}(r, \cdot,\cdot) = f(r,\cdot,\cdot)$
for a.e.\ $r\in [0,R]$
(see \cite[Definition~VIII.1.1]{EkTem}).

\section{Symmetry of minimizers}
\label{s:radial}

In this section we deal with the symmetry properties of minimizers in
$W^{1,1}_0(B_R)$ of functionals of the form
\[
F(u) := \int_{B_R} f(|x|, u, |\nabla u|)\, dx
\]
under very mild assumptions on the Lagrangian $f$.

Our aim is to prove that the minimization problem for $F$ in $W^{1,1}_0(B_R)$
is, in fact, equivalent to the minimization problem for the one--dimensional
functional
\begin{equation}\label{f:frad}
\Frad (u) := \int_0^R r^{N-1}\, f(r, u(r), |u'(r)|)\, dr,
\end{equation}
in the functional space
\begin{equation}\label{f:wbase}
\Wbase := \left\{
u \in AC_{\textrm{loc}}(]0,R]):\
u(R) = 0,\
r^{N-1}\, |u'(r)|\in L^1(]0,R[)
\right\}\,.
\end{equation}

\begin{remark}
Notice that the functional $\Frad$ is, up to a constant factor, the
functional $F$ evaluated on the radially symmetric functions belonging to
$W^{1,1}_0(B_R)$. In particular, we underline that every function 
$u\in\Wbase[1]$ satisfies
\[
r^{N-1}|u(r)| 
\leq \int_r^R \rho^{N-1}  |u'(\rho)|\, d\rho \leq \|\rho^{N-1}u'(\rho)\|_{L^1}
\qquad
\forall r\in ]0, R],
\]
so that $r^{N-1}|u(r)| \in L^\infty([0,R])$.
\end{remark}

We adopt a symmetrization procedure introduced in \cite{Kro}.
Given a representative of  $u\in W^{1,1}_0(B_R)$, and
$\theta\in \partial B_1$, let 
\begin{equation}\label{f:uteta}
u_\theta(x) := u(\theta |x|), \qquad x\in B_R\,,
\end{equation}
be the radial symmetric function obtained from the profile of $u$
along the straight line through $0$ and with direction $\theta$. 

In \cite[Lemma~3.1]{Kro} it is proved that  $u_\theta \in W^{1,1}_0(B_R)$ for 
a.e.\ $\theta\in \partial B_1$, and 
\begin{equation}\label{f:gradut}
|\nabla u_\theta(x)| = \left| \theta \cdot \nabla u(\theta |x|)
\, \frac{x}{|x|}
\right|
\leq |\nabla u(\theta |x|)|\,.
\end{equation} 

Following the lines of the proof of \cite[Theorem~3.4]{Kro}, we show that,
for some $\theta$,
$u_\theta$ is a better competitor than $u$ in the minimization problem for $F$.

\begin{theorem}\label{t:radial}
Let $f\colon [0,R]\times\R\times [0,+\infty[ \to ]{-}\infty, +\infty]$ be a
normal integrand such that for almost every $(r, t)\in [0,R]\times\R$,
the map $s\mapsto f(r,t,s)$ is monotone non-decreasing.
Then for every $u\in W^{1,1}_0(B_R)$ there exists a radially symmetric function
$v\in W^{1,1}_0(B_R)$ such that $F(v) \leq F(u)$.
In particular, if $F$ admits minimizers in $W^{1,1}_0(B_R)$,
then it admits a radially symmetric minimizer.

If, in addition,
for almost every $(r, t)\in [0,R]\times\R$,
the map $s\mapsto f(r,t,s)$ is strictly monotone increasing,
then every minimizer of $F$ in $W^{1,1}_0(B_R)$ is a
radially symmetric function.
\end{theorem}

\begin{proof}
Let $u$ be a function in $W^{1,1}_0(B_R)$ such that $F(u)<+\infty$, and let 
$u_\theta$ be the 
radially symmetric function defined in \eqref{f:uteta}.
We claim that, 
\begin{equation}\label{f:ine1}
\frac{1}{\mathcal{H}^{N-1}(\partial B_1)} \int_{\partial B_1} F(u_\theta)\, 
d\theta
\leq F(u)\,,
\end{equation}
where $\haus$ is the $(N-1)$-dimensional Hausdorff measure.
Namely, observing that  
\[
u_\theta(r\omega)=u_\theta(r\theta)=u(r\theta)\,,
\qquad
\forall \omega, \theta\in \partial B_1,
\]
using \eqref{f:gradut} and the 
monotonicity property of the Lagrangian $f$, 
we obtain that
\[
\begin{split}
\frac{1}{\mathcal{H}^{N-1}(\partial B_1)}
\int_{\partial B_1} F(u_\theta)\, d\theta & =
\int_{\partial B_1} \int_{\partial B_1}
\int_0^R f(r, u_\theta(r\omega), |\nabla u_\theta(r\omega)|)\,
r^{N-1}\, dr\, d\omega\, d\theta
\\ & =
\int_{\partial B_1} \int_{\partial B_1}
\int_0^R f(r, u_\theta(r\theta), |\nabla u_\theta(r\theta)|)\,
r^{N-1}\, dr \, d\omega\, d\theta
\\ & \leq
\int_{\partial B_1} \int_{\partial B_1}
\int_0^R f(r, u(r\theta), |\nabla u(r\theta)|)\,
r^{N-1}\, dr\,  d\omega\, d\theta
=   F(u).
\end{split}
\]
From \eqref{f:ine1} follows that there exists a set $\Theta \subseteq \partial 
B_1$, with $\mathcal{H}^{N-1}(\Theta)>0$, such that $F(u_\theta) \leq F(u)$
for every $\theta\in\Theta$. Moreover, if $u$ is a minimizer for $F$, then
$F(u_\theta) \geq F(u)$ for a.e.\ $\theta \in \partial B_1$, and
\eqref{f:ine1} implies that
\begin{equation}\label{f:eq1}
F(u_\theta) = F(u)
\qquad
\text{for $\mathcal{H}^{N-1}$--a.e.}\ \theta \in \partial B_1,
\end{equation}
hence almost every $u_\theta$ is a (radially symmetric)
minimizer of $F$.

\smallskip

Assume now that for almost every $(r, t)\in [0,R]\times\R$,
the map $s\mapsto f(r,t,s)$ is strictly monotone increasing, and let
$u$ be a minimizer for $F$.  
From the computation above,
we deduce that \eqref{f:eq1} holds if and only if
\[
f(r, u_\theta(r\theta), |\nabla u_\theta(r\theta)|)
=
f(r, u(r\theta), |\nabla u(r\theta)|)
\quad
\text{for $\mathcal{L}\times\mathcal{H}^{N-1}$--a.e.}\ (r,\theta)\in 
[0,R]\times \partial B_1.
\]
Since $u_\theta(r\theta) = u(r\theta)$ for a.e.\ $(r, \theta)$,
from the strict monotonicity assumption on $f$ we deduce that
$ |\nabla u_\theta(r\theta)| = |\nabla u(r\theta)|$
for $\mathcal{L}\times\mathcal{H}^{N-1}$-a.e.\ $(r, \theta)$,
hence, from \eqref{f:gradut}, we obtain that $\nabla u(r\theta)$ is parallel to 
$\theta$ and then
$u$ is radially symmetric
(see \cite[Lemma~3.3]{Kro}).
\end{proof}

As a consequence of Theorem \ref{t:radial}, we obtain the following 
1--dimensional reduction of the minimum problem.

\begin{corollary}\label{c:radial}
Let $f$ be as in Theorem~\ref{t:radial}.
Then the minimization problem
\begin{equation}\label{f:minJ}
\min\{F(u):\ u\in W^{1,1}_0(B_R)\}
\end{equation}
admits a solution if and only if 
the one-dimensional minimization problem
\begin{equation}\label{f:minR}
\min\{\Frad(u):\ u\in \Wbase[1]\}
\end{equation}
admits a solution, where $\Frad$ and $\Wbase[1]$ are defined in 
\eqref{f:frad} and \eqref{f:wbase} respectively. 
\end{corollary}

\begin{proof}
If problem \eqref{f:minJ} admits a solution $u\in W^{1,1}_0(B_R)$,
then by Theorem~\ref{t:radial} there exists a 
radially symmetric function $v\in W^{1,1}_0(B_R)$ such that $F(v) \leq F(u)$,
hence $\rad[v] (r) := v(|x|)$ is a solution to problem~\eqref{f:minR}.

Assume now that problem~\eqref{f:minR} admits a solution $\rad[u]\in\Wbase[p]$,
and let us prove that $u(x) := \rad[u](|x|)$ is a solution
to~\eqref{f:minJ}.
Namely, if we assume by contradiction that there exists a function
$v \in W^{1,1}_0(B_R)$ such that $F(v) < F(u)$,
then by Theorem~\ref{t:radial} there exists a radially symmetric
function $w\in W^{1,1}_0(B_R)$ such that $F(w) \leq F(v)$, so that
the function $\rad[w](r) := w(|x|)$ satisfies
$\Frad(\rad[w]) <  \Frad(\rad[u])$, a contradiction.
\end{proof}

\section{Existence of minimizers and Euler--Lagrange inclusions}
\label{s:min}

In this section
we focus our attention to functionals of the form
\[
F(u) := \int_{B_R} \left[
g(|x|, |\nabla u|) + h(|x|,u)
\right]\,dx,
\qquad u\in W^{1,1}_0(B_R),
\]

whose corresponding one-dimensional functional is
\[
\Frad(u) := \int_0^R r^{N-1} [g(r, |u'(r)|) + h(r, u(r))]\, dr,
\qquad
u\in\Wbase.
\]

We prove  the existence of radially symmetric
Lipschitz continuous minimizers, and the validity of necessary optimality conditions
of Euler--Lagrange type, when $g$ is a convex function with possibly linear growth
in the gradient variable, and 
$h$ is a Lipschitz continuous function with respect to $u$.    

As usual, the Euler--Lagrange conditions
involve a pair $(\rad,p)$, where $\rad$ is a minimizer in $\Wbase$, 
while the function $p$ belongs to the space
\[
\Wbasestar := \left\{
p\in AC([0,R]):\
p(0) = 0,\
r^{1-N} p'(r) \in L^1(]0,R[)
\right\}\,.
\]
We call $p$ a \textsl{momentum} associated with $\rad$.

\begin{theorem}
\label{t:noncoercive}
Let $g\colon [0,R]\times [0,+\infty[ \to [0,+\infty]$, and $h\colon 
[0,R]\times\R\to\R$ satisfy:
\begin{itemize}
\item[(g1r)]
$g$
is a normal integrand, the function 
$z \mapsto g(r, |z|)$ is convex for a.e.\ $r\in [0,R]$,
and $r^{N-1} g(r, 0) \in L^1(]0,R[)$.

\item[(g2r)]
There exists a function $\psi\colon [0, +\infty[\to [0, +\infty[$
such that
\[
\text{for a.e.}\ r\in [0,R]:\quad
g(r,s) - g(r,0) \geq \psi(s)
\quad
\forall s\geq 0, 
\]
and $M := \liminf_{s\to +\infty} {\psi(s)}/{s} > 0$.

\item[(h1r)]
$h$ is a Borel function,
$r^{N-1} h(r, 0) \in L^1(]0,R[)$,
and there exists $H_0 \in L^1(]0,R[)$ such that
\[
\text{for a.e.}\ r\in [0,R]:
\quad
|h(r, t) - h(r, \tau)| \leq H_0(r)\, |t-\tau|
\qquad
\forall t, \tau\in\R. 
\]

\item[(hgr)]
The functions $g$ and $h$ are related by the condition
\[
M_0 := 
\sup_{r\in ]0,R]} r^{1-N} \int_0^r \rho^{N-1} H_0(\rho)\, d\rho
< M.
\]
\end{itemize}
Then the following holds true.
\begin{itemize}
\item[(i)]
$F$ admits a radially symmetric minimizer in $W^{1,1}_0(B_R)$, and $\Frad$
admits a minimizer in $\Wbase$.
\item[(ii)]
Every minimizer of $\Frad$ is Lipschitz continuous.
\item[(iii)]
For every minimizer $\rad\in\Wbase$ of $\Frad$
there exists $p\in \Wbasestar$ such that
the following Euler--Lagrange inclusions hold:
\begin{gather}
p'(r) \in r^{N-1} \partial h(r, \rad(r)),
\qquad\text{for a.e.}\ r\in [0,R],
\label{f:EL1}\\
p(r) \in r^{N-1} \partial g(r, |\rad'(r)|), 
\qquad\text{for a.e.}\ r\in [0,R].
\label{f:EL2}
\end{gather}
\end{itemize}
\end{theorem}

\begin{remark}
In (g2r) 
it is not restrictive to assume that 
$\psi$ is a non-decreasing function, with $\psi(0) = 0$,
and that $\R\ni z \mapsto \psi(|z|)$ is convex and smooth (possibly
replacing $\psi$ with a suitable regularization of its convex envelope).
As a consequence of these assumptions, the function $s\mapsto \psi(s) 
/ s$ turns out to be strictly increasing
in $]s_0,+\infty[$, where $s_0 := \max\{\psi = 0\}$,
and hence, for every $m\in ]0,M[$, there exists (a unique) $\sigma > s_0$
such that $\psi(\sigma) / \sigma = m$. 
In the following we shall always assume that the function $\psi$
in (g2r) satisfies these additional properties.
We recall that, if $M=+\infty$, such a function 
is called a \textsl{Nagumo function} (see, e.g., \cite[Section 10.3]{Ces}).
\end{remark}

\begin{remark}
If $g$ satisfies (g1r) and (g2r), then
\begin{gather}
]{-}M, M[ \subset \Dom g^*(r, \cdot),
\qquad
\text{for a.e.}\ r\in [0,R],
\label{f:domgs}
\\
r^{N-1} g^*(r, m) \in L^1(]0,R[),
\qquad \forall m\in ]{-}M, M[. 
\label{f:intg}
\end{gather}
Specifically, by symmetry it is enough to show that,
for every $m\in ]0,M[$,
$m\in \Dom g^*(r, \cdot)$ for a.e.\ $r\in[0,R]$
and \eqref{f:intg} holds.
Let $m\in ]0,M[$ and let $\sigma>0$ satisfy $\psi(\sigma)/\sigma = m$.
Then
\[
\frac{g(r,s) - g(r, 0)}{s} \geq \frac{\psi(s)}{s}
\geq \frac{\psi(\sigma)}{\sigma} = m,
\qquad \forall s\geq \sigma,
\]
so that 
$-g(r,0)\leq g^*(r,m) = \sup_{s\geq 0} [m\,s - g(r,s)] \leq m\,\sigma - g(r, 0)$.
Hence, \eqref{f:domgs} and~\eqref{f:intg} follow from the assumption
$r^{N-1} g(r,0) \in L^1(]0,R[)$.
\end{remark}

\begin{remark}
\label{r:H0}
If $h$ satisfies (h1r), then the quantity $M_0$ defined in (hgr) is always finite, since
\[
r^{1-N} \int_0^r \rho^{N-1} H_0(\rho)\, d\rho
\leq \int_0^r  H_0(\rho)\, d\rho
\leq \|H_0\|_{L^1}\,,
\qquad \forall r\in ]0,R].
\]
\end{remark}

\medskip
We start by proving some \textit{a-priori} estimates for the
solutions of the Euler--Lagrange inclusions. 

\begin{lemma}
\label{l:EL}
Let $(\rad, p)\in\Wbase\times\Wbasestar$.
Then the following hold:
\begin{itemize}
\item[(i)]
If $h$ satisfies (h1r) 
and $(\rad, p)$ satisfies \eqref{f:EL1}, then
$r^{1-N} |p(r)| \leq M_0$ for every $r\in ]0,R]$,
where $M_0$ is the (finite) quantity defined in (hgr).
\item[(ii)]
If $g$ and $h$ satisfy (g1r)-(g2r)-(h1r)-(hgr),
and the pair $(\rad, p)$ satisfies
the Euler--Lagrange inclusions \eqref{f:EL1}--\eqref{f:EL2},
then
\begin{equation}
\label{f:apriori}
|\rad'(r)| \leq
\sigma(r) := g_+^{*'}(r, M_0),
\qquad \text{for a.e.}\ r\in [0,R].
\end{equation}
Moreover, if $\sigma_0>0$ is defined by
\begin{equation}
\label{f:sigma0}
\frac{\psi(\sigma_0)}{\sigma_0} = M_0,
\end{equation}
then $\sigma(r) \leq \sigma_0$ for a.e.\ $r\in [0,R]$,
i.e., $\rad$ is Lipschitz continuous and
\begin{equation}\label{f:lip}
|\rad'(r)| \leq \sigma_0,
\qquad \text{for a.e.}\ r\in [0,R].
\end{equation}
\end{itemize}
\end{lemma}

\begin{proof}
(i) 
From Remark~\ref{r:H0},
the quantity $M_0$ defined in (hgr) is finite.
By (h1r) we have that $\partial h(r,t) \subseteq [{-}H_0(r),H_0(r)]$
for a.e.\ $r\in [0,R]$, so that
\eqref{f:EL1} gives the estimate
\[
|p'(r)| \leq r^{N-1} H_0(r) 
\qquad
\text{for a.e.}\ r\in [0,R],
\]
and hence
\begin{equation}
\label{f:pH}
\sup_{r\in ]0,R]} r^{1-N} |p(r)|
\leq
\sup_{r\in ]0,R]} r^{1-N} \int_0^r \rho^{N-1} H_0(\rho)\, d\rho = M_0.
\end{equation}

(ii) From \eqref{f:EL2} we have that $|u'(r)| \in \partial g^{*}(r, 
r^{1-N}p(r))$, and, from \eqref{f:pH}, we deduce that
\[
|\rad'(r)| \leq g_+^{*'}(r, r^{1-N}p(r))
\leq g_+^{*'}(r, M_0)
\qquad
\text{for a.e.}\ r\in [0,R],
\]
so that \eqref{f:apriori} holds.
Moreover, if $\sigma_0$ is defined by \eqref{f:sigma0},
then, by the convexity assumption on $g(r,\cdot)$, we obtain the estimate
\[
g'_-(r, \sigma_0) \geq M_0
\qquad
\text{for a.e.}\ r \in [0,R]
\]
(with the convention $g'_-(r, \sigma_0) = +\infty$ if
$\sigma_0\not\in\Dom g(r, \cdot)$).
On the other hand, by the very definition of $\sigma(r)$, we have that
$M_0\in \partial g (r, \sigma(r))$, hence 
\[
g'_-(r, \sigma_0) \geq M_0 \geq g'_-(r, \sigma(r))
\qquad
\text{for a.e.}\ r \in [0,R]\,,
\] 
which in turn implies that 
$
\sigma(r) \leq \sigma_0
$
for a.e.\ $r\in [0,R]$,
and \eqref{f:lip} follows.
\end{proof}

The proof of Theorem~\ref{t:noncoercive} is divided into two steps: first
we show that the result is valid in the superlinear case, i.e.\ when $M=+\infty$, 
and then we obtain the result  when $M<+\infty$ by constructing, with the help 
of 
the \textit{a-priori} estimates obtained by the Euler--Lagrange conditions, a family
of superlinear functional whose radially symmetric minimizers also minimize  
the functional $F$.

\begin{proof}[Proof of Theorem~\ref{t:noncoercive}]
\textsl{Step~1: superlinear Lagrangians.}

(i) In order to use a standard existence result for coercive functionals
(see, e.g., \cite[Theorem~2.2]{EkTem}),
we need to rewrite the functional $F$ in a suitable form.

Let us define
\begin{gather*}
P(r) := \int_0^r \rho^{N-1} H_0(\rho)\, d\rho,
\qquad 
G(r,s) := 
g(r,s) + r^{1-N} P(r) \, s\,,
\\
H(r,t) := h(r,t) - h(r,0) + H_0(r) |t|
=  h(r,t) - h(r,0) + r^{1-N} P'(r) |t|. 
\end{gather*}
Since, by (h1r), it holds that
\[
h(r,t) \geq h(r,0) - H_0(r) |t|,
\qquad
\forall r\in [0,R],\
t\in\R,
\]
then $H(r,t)\geq 0$ for all $r\in [0,R]$ and 
$t\in\R$.
Moreover, we have that
\[
\begin{split}
\Frad(u) = {} &
\int_0^R r^{N-1}\left[
g(r, |u'|) + h(r,u) - h(r,0) + H_0(r) |u|
\right]\, dr
\\ & 
- \int_0^R P'(r)\, |u|\, dr
+ \int_0^R r^{N-1} h(r,0)\, dr.
\end{split}
\]
Since $(|u|, P) \in \Wbase\times\Wbasestar$, it holds that
\[
\int_0^R P'(r)\, |u|\, dr = -\int_0^R P(r) |u|'\, dr =
-\int_0^R P(r) |u'|\, dr
\]
(see, e.g., the derivation of formula (13) in \cite{C4}).
Setting $C := \int_0^R r^{N-1} h(r,0)\, dr$, we get
\[
\Frad(u) = \int_0^R r^{N-1} \left[
G(r, |u'|) + H(r, u)
\right]\, dr + C.
\]

Observe that, by (g2r),
\[
G(r, s) + H(r, t) \geq G(r,s) \geq \psi(s) -M_0\, s + g(r,0).
\]
Since $\psi$ is a Nagumo function,
then by Theorem~2.2 in \cite{EkTem} the functional
\[
\widehat{F}(u) := \int_{B_R} [G(|x|, |\nabla u|) + H(|x|, u)] \, dx
\]
admits a minimizer 
in $W^{1,1}_0(B_R)$.
Hence, by Corollary~\ref{c:radial}, the functional $\Frad$ admits a
minimizer in $\Wbase$.

\medskip
(ii)-(iii) Let us prove that, for every minimizer $\rad$ of $F$ in $\Wbase$,
there exists a momemtum $p\in\Wbasestar$
associated with $\rad$. 
(Hence, the Lipschitz continuity of $\rad$ will follow from 
Lemma~\ref{l:EL}.)
Specifically, the conclusion follows
from \cite[Theorem~4.2.2]{Clar},
once we show that all the assumptions are satisfied. 
The Lagrangian 
$L(r, t, s) := r^{N-1} [ g(r, |s|) + h(r, t)]$
is convex 
with respect to $s$, and
satisfies the \textsl{Basic Hypotheses} 4.1.2 in \cite{Clar}.
Moreover, the Hamiltonian of the problem,
i.e., the Fenchel--Legendre transform of $L$
with respect to the last variable:
\[
H(r,t,p) := \sup_{s\in\R} [p\, s - L(r,t,s)]
= r^{N-1} [g^*(r, r^{1-N} p) - h(r,t)],
\quad
\forall (r,t,p) \in ]0, R]\times\R\times\R,
\]
satisfies the \textsl{strong Lipschitz condition} near every arc,
since, by (h1r),
\[
|H(r, t, p) - H(r, \tau, p)| 
= r^{N-1} |h(r,t) - h(r, \tau)|
\leq r^{N-1} H_0(r)\, |t - \tau|.
\]
Finally, the minimization problem is \textsl{calm}, since it is a
free-endpoint problem,
hence all assumptions of Theorem~4.2.2 in \cite{Clar}
are satisfied.

\medskip
\textsl{Step~2: slow growth Lagrangians.}

(i) Let $\sigma_0>0$ be defined by \eqref{f:sigma0}, and, for $a > \sigma_0$ 
given,
let $\Phi_a$ be the class of all convex superlinear non-decreasing functions
$\varphi\colon [0,+\infty[\to [0, +\infty[$,
such that $\varphi(s) = 0$ for every $s\in [0,a]$.

Given $\lambda \geq 0$ and $\varphi\in\Phi_a$, let us define the superlinear
Lagrangian
\[
\glambda(r, s) := g(r,|s|) + \lambda\, \varphi(|s|)
\]
and the corresponding functional
\begin{equation}
\begin{split}
\Flambda(u)  := {} & \int_0^R r^{N-1} [\glambda(r, |u'|) + h(r,u)]\, dr
\\ = {} &
\Frad(u) + \lambda \int_0^R r^{N-1} \varphi(|u'(r)|)\, dr\,,
\qquad u\in\Wbase\,.
\end{split}
\label{f:Fl}
\end{equation}
For every $\lambda > 0$ and $\varphi\in\Phi_a$
the functional $\Flambda$ satisfies the assumptions of Step~1,
hence there exist a minimizer
$\ulambda$ of $\Flambda$ in $\Wbase$ and an associated momentum
$\plambda\in\Wbasestar$, such that
\[
\begin{split}
\plambda'(r) & \in r^{N-1} \partial h(r, \ulambda(r)),
\qquad\text{for a.e.}\ r\in [0,R],
\\
\plambda(r) & \in r^{N-1} \partial \glambda(r, |\ulambda'(r)|), 
\qquad\text{for a.e.}\ r\in [0,R].
\end{split}
\]
By Lemma \ref{l:EL}(i), we obtain that $r^{1-N}|\plambda(r)| \leq M_0$
for every $r\in ]0,R]$. On the other hand, since
\[
r^{1-N}\plambda(r)  \in  \left[(\glambda)_-'(r, |\ulambda'(r)|), 
(\glambda)_+'(r, |\ulambda'(r)|)\right]
\]
and, by Lemma \ref{l:EL}(ii),  $M_0\in \partial g(r, \sigma(r))$ with
$\sigma(r)\leq \sigma_0<a$, we obtain that 
\[
g_-'(r, |\ulambda'(r)|)\leq
(\glambda)_-'(r, |\ulambda'(r)|)\leq M_0 \leq g_+'(r, \sigma(r)) \leq
g_-'(r, a).
\]
Hence, $|\ulambda'|\leq a$ a.e.\ in $[0,R]$, so that $\varphi(|\ulambda'|)=0$,
and $\Flambda(\ulambda) = F(\ulambda)$.

By the discussion above,
for every $\varphi \in \Phi_a$ and every $\mu\geq\lambda > 0$,
we have that
\[
\Flambda(\ulambda[\mu]) \geq
\Flambda(\ulambda) =
\Flambda[\mu](\ulambda) 
\geq
\Flambda[\mu](\ulambda[\mu])
\geq
\Flambda(\ulambda[\mu]),
\]
hence we conclude that $m := F(\ulambda)$ is independent of $\lambda > 0$ and $\varphi\in\Phi_a$.

We claim that $m = \min_{\Wbase} F$.
Specifically,
assume by contradiction that there exists $v\in\Wbase$ such that
$\Frad(v) < m$.
Since $|\nabla v(|x|)| \in L^1(B_R)$,
by the de La Vallée Poussin criterion (see, e.g. \cite[Theorem 10.3.i]{Ces}), 
there exists a
function $\varphi\in\Phi_a$ such that 
$\int_{B_R} \varphi(|\nabla v(|x|)|)\, dx < +\infty$, i.e.
\[
\int_0^R r^{N-1} \varphi(|v'(r)|)\, dr < +\infty.
\]
By \eqref{f:Fl}, for $\lambda > 0$ small enough we have that
$\Flambda(v) < m = \min \Flambda$,
a contradiction.

\medskip
(ii) Let $\rad$ be a minimizer of $F$ in $\Wbase$,
and let us prove that $\rad$ is Lipschitz continuous.

Assume by contradiction that $\rad$ is not Lipschitz continuous,
i.e.  
$\mathcal{L}(\{|\rad'| > a\}) > 0$ for every $a > 0$ (here
$\mathcal{L}$ denotes the Lebesgue measure on $\R$).

Let us define $\delta$, $\widehat{\sigma}$ and $\sigma_1$ by:
\[
\delta := \frac{M - M_0}{3}\,,
\qquad
\widehat{\sigma}(r) := g_-^{*'}(r, M_0 + \delta),
\qquad
\frac{\psi(\sigma_1)}{\sigma_1} = M_0 + 2\delta.
\]
Observe that, by (g2r),
\[
g'_-(r, \sigma_1) \geq \frac{\psi(\sigma_1)}{\sigma_1} =  M_0 + 2\delta > M_0 + 
\delta
\geq g'_-(r, \widehat{\sigma}(r)),
\]
so that $\sigma_1 \geq \widehat{\sigma}(r)$ for every $r\in [0,R]$.
(The inequality is trivially satisfied for those values of $r$
such that $\sigma_1\not\in \Dom g(r, \cdot)$.)
Let us define the function
\[
\ell(r,s) := g(r, \widehat{\sigma}(r)) + (M_0 + \delta) (s - 
\widehat{\sigma}(r)),
\qquad r\in [0,R],
\ s\geq 0.
\]
Since $M_0+\delta\in\partial g(r, \widehat{\sigma}(r))$, we have that
$g(r,s) \geq l(r,s)$ for every $r\in[0,R]$ and $s\geq 0$.

Let $\varphi$ be a Nagumo function
such that $\int_0^R r^{N-1} \varphi(|\rad'|)\, dr < +\infty$.
Given $a > 0$, let $\varphi_a := [(\varphi - \varphi(a)) \vee 0] \in \Phi_a$.
Since $0\leq\varphi_a\leq\varphi$,
we have that
\[
0\leq \lim_{a\to +\infty}\int_{\{|\rad'| > a\}} r^{N-1} \varphi_a(|\rad'|)\, dr
\leq \lim_{a\to +\infty}
\int_{\{|\rad'| > a\}} r^{N-1} \varphi(|\rad'|)\, dr
= 0,
\]
whereas
\[
\lim_{a\to +\infty} \int_{\{\sigma_1 \leq |\rad'| \leq a\}} 
r^{N-1} (|\rad'| - \sigma_1)\, dr =
\int_{\{\sigma_1 \leq |\rad'| \}} 
r^{N-1} (|\rad'| - \sigma_1)\, dr >0\,,
\]
hence there exists $\zeta > \sigma_1$ such that
\begin{equation}
\label{f:diseq}
\delta
\int_{\{\sigma_1 \leq |\rad'| \leq \zeta\}} 
r^{N-1} (|\rad'| - \sigma_1)\, dr
>
\int_{\{|\rad'| > \zeta\}} r^{N-1} \varphi_\zeta(|\rad'|)\, dr\,.
\end{equation}

\begin{figure}
\includegraphics{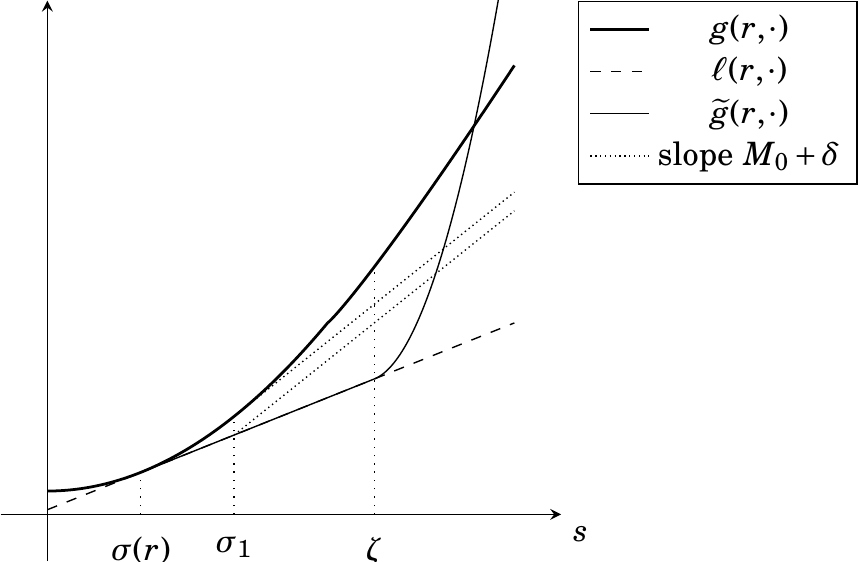}
\caption{Construction of $\gtilde$}
\label{fig1}
\end{figure}

For every $r\in[0,R]$, let us define the function
(see Figure~\ref{fig1})
\[
\gtilde (r,s) :=
\begin{cases}
g(r,s), &\text{if}\ 0\leq s \leq \widehat{\sigma}(r),
\\
\ell(r,s) + \varphi_\zeta(s),
&\text{if}\ s\geq \widehat{\sigma}(r), 
\end{cases}
\]
and let
\[
\Ftilde(v) := \int_0^R r^{N-1} [\gtilde(r, |v'|) + h(r,v)]\, dr\, ,
\qquad v\in \Wbase.
\]

Since $g'_-(r, \sigma_1) \geq M_0+2\delta$, 
for every $s\in [\sigma_1, \zeta]$
we have that
\begin{equation}\label{f:estig}
\begin{split}
g(r, s) & \geq g(r, \sigma_1) + (M_0+2\delta) (s-\sigma_1)
\\ & \geq \gtilde(r, \sigma_1) 
+ (M_0+\delta) (s-\sigma_1)
+ \delta (s-\sigma_1)
\\ & = \gtilde(r, s) 
+ \delta (s-\sigma_1).
\end{split}
\end{equation}
Observe that, by the definition of $\gtilde$
and \eqref{f:estig},
\begin{align*}
&\gtilde(r, |\rad'|) \leq g(r, |\rad'|),
& &\text{a.e.\ in $\{|\rad'| < \sigma_1\}$},
\\
&\gtilde(r, |\rad'|) \leq g(r, |\rad'|) - \delta(|\rad'| - \sigma_1),
& &\text{a.e.\ in $\{\sigma_1 \leq |\rad'| \leq  \zeta\}$},
\\
&\gtilde(r, |\rad'|) \leq g(r, |\rad'|) + \varphi_\zeta(|\rad'|),
& &\text{a.e.\ in $\{|\rad'| > \zeta\}$},
\end{align*}
hence, by \eqref{f:diseq},
\[
\Ftilde(\rad) \leq \Frad(\rad)
- \delta
\int_{\{\sigma_1 \leq |\rad'| \leq \zeta\}} 
r^{N-1} (|\rad'| - \sigma_1)\, dr
+
\int_{\{|\rad'| > \zeta\}} 
r^{N-1} \varphi_\zeta(|\rad'|)\, dr
< \Frad(\rad)\,.
\]
On the other hand, if $\widetilde{u}$  is a minimizer of $\Ftilde$,
then by Step~1 there exists $p\in\Wbasestar$ such that 
$(\widetilde{u}, p)$ satisfies the Euler--Lagrange inclusions~\eqref{f:EL1}--\eqref{f:EL2} with $g$ replaced by $\gtilde$.
From Lemma~\ref{l:EL}(i)
we deduce that 
\[
|\widetilde{u}'(r)| \leq
\gtilde_+^{*'}(r, r^{1-N} p(r))
\leq
\gtilde_-^{*'}(r, M_0+\delta)
\leq \widehat{\sigma}(r),
\qquad
\text{for a.e.}\ r\in [0,R],
\] 
(where the last inequality follows from
$\gtilde'(r, \widehat{\sigma}(r)) = M_0+\delta$),
hence
\[
\gtilde(r, |\widetilde{u}'|)
= g(r, |\widetilde{u}'|),
\qquad
\text{for a.e.}\ r\in [0,R],
\]
and, in conclusion,
\[
\Frad(\widetilde{u}) = \Ftilde(\widetilde{u})
\leq \Ftilde(\rad) < \Frad(\rad),
\]
in contradiction with the assumption that $\rad$ is a minimizer of $F$.

\medskip
(iii) Finally, let us prove that $\rad$ satisfies the
Euler--Lagrange inclusions.
Let $\sigma > 0$ be
such that $|\rad'| \leq \sigma$ a.e.\ in $[0,R]$.
Reasoning as in the existence proof, 
$\rad$ is a minimizer of $\Flambda$
for every $\lambda> 0$ and $\varphi\in \Phi_a$, with $a > \sigma\vee\sigma_0$.
Hence, $\rad$ satisfies the Euler--Lagrange inclusions with
$\glambda$ instead of $g$. 
Since $\partial\glambda(r, |\rad'|) = \partial g(r, |\rad'|)$
for a.e.\ $r\in [0,R]$, the conclusion follows.
\end{proof}

\section{Convex solutions of variational problems with gradient constraints}
\label{s:convex}

As an application of the previous results, we obtain the existence of convex
radially symmetric minimizers for autonomous functionals of the form
\begin{equation}\label{f:auton}
F(u) := \int_{B_R} \left[
g(|\nabla u|) + h(u)
\right]\,dx,
\end{equation}
in the space
\[
\W := \left\{
u \in W^{1,1}_0(\Omega):\
|\nabla u(x)| \leq \mu(|x|)\ \text{for a.e.}\ x\in B_R
\right\}
\]
of Sobolev functions with gradient constraint given by a monotone
non-decreasing function
$\mu\colon [0,R] \to ]0,+\infty]$.

\begin{theorem}\label{t:convex}
Let us consider the integral functional \eqref{f:auton},
where $g\colon [0,+\infty[\to\R$
and $h\colon\R\to\R$
satisfy the following
assumptions:
\begin{itemize}
\item[(g1)]
$\R\ni z\mapsto g(|z|)$ is a convex function;
\item[(g2)] $M := \lim_{s\to +\infty} g(s) / s > 0$;
\item[(h1)] 
$h$ is a convex function;
\item[(hg)] $\min\{ |h'_-(0)|, |h'_+(0)|\} < \frac{N M}{R}$.
\end{itemize}
Then the following hold.
\begin{itemize}
\item[(i)] 
$F$ admits a radially symmetric minimizer $u(x) = \rad(|x|)$ in $\W$.
\item[(ii)] There exists a momentum $p\in\Wbasestar$ such that
the following Euler--Lagrange inclusions hold:
\begin{gather}
p'(r) \in r^{N-1} \partial h(\rad(r)),
\qquad\text{for a.e.}\ r\in [0,R],
\label{f:EL1t}\\
p(r) \in r^{N-1} \Gamma(r, |\rad'(r)|)
\qquad\text{for a.e.}\ r\in [0,R],
\label{f:EL2t}
\end{gather}
where
\[
\Gamma(r, s) :=
\begin{cases}
\partial g(s), &\text{if}\ 0\leq s < \mu(r),\\
[\mu(r), +\infty[, &\text{if}\ s = \mu(r),\\
\emptyset, &\text{if}\ s > \mu(r).
\end{cases}
\]

\item[(iii)]
If $h'_+(0) \geq 0$ [resp.\ $h'_-(0) \leq 0$],
then $u$ is a convex [resp.\ concave] function.
\item[(iv)]
If, in addition, 
$g$ has a strict minimum point at $0$,
or $h$ is a strictly monotone function,
then every minimizer of $F$ in $\W$
is radially symmetric. 
\end{itemize}
\end{theorem}

\begin{proof}
The constraint $|\nabla u(x)|\leq\mu(|x|)$ in the definition of the
functional space $\W$ can be incorporated into the Lagrangian.
Specifically, let us define
\[
\gtilde(r,s) := g(s) + \one_{[0, \mu(r)]}(s),
\qquad
\widetilde{F}(u) := \int_{B_R} [\gtilde(|x|, |\nabla u(x)|) + h(u(x))]\, dx,
\] 
where $\one_B$ is the indicator function of a set $B$, defined by
$\one_B(s) = 0$ if $s\in B$ and $+\infty$ otherwise.
Then minimizing $F$ in $\W$ is equivalent to minimizing
$\widetilde{F}$ in $W^{1,1}_0(B_R)$.

We remark that, if $g$ satisfies (g1)--(g2), then
$\gtilde$ satisfies (g1r)--(g2r)
and
\begin{equation}\label{f:gtsub}
\partial\gtilde(r, s) = \Gamma(r, |s|),
\qquad \forall (r,s)\in [0,R]\times\R.
\end{equation}

\medskip
We shall prove the theorem only in the case $h'_+(0) \geq 0$ 
(since the case $h'_-(0) \leq 0$ can be handled similarly).

If $0$ is a minimum point of $h$, then clearly
parts (i)-(ii)-(iii) are satisfied choosing $u \equiv 0$
and $p\equiv 0$.
Hence, it is not restrictive to prove (i)-(ii)-(iii) under
the additional assumption that $0$ is not a
minimum point of $h$.
Since $h'_+(0) \geq 0$,
and $h$ is a convex function,
we have that $h'_+(0)\geq h'_-(0) > 0$.

Since $h'_-(0) > 0$,
the (possibly empty) convex and closed set $\argmin h$ is 
contained in the open half-line $]{-}\infty, 0[$.
If $\argmin h \neq \emptyset$, let $m := \max \argmin h$,
otherwise let $m = {-}\infty$.
Let us define
\[
\htilde(t) := 
\begin{cases}
h(m), & \text{if}\ t \leq m,\\
h(t), & \text{if}\ m < t \leq 0,\\
h(0) + h'_-(0)\, t,
& \text{if}\ t > 0,
\end{cases}
\]
(the first condition is empty if $m=-\infty$) and
\[
\Fhat(u) := \int_{B_R} [\gtilde(|x|, |\nabla u|) + \htilde(u)]\, dx,
\qquad
u \in W^{1,1}_0(B_R).
\]
Given $v\in W^{1,1}_0(B_R)$, let $v_m := (v\wedge 0) \vee m$,
and observe that
$\Fhat(v_m) \leq \Fhat(v)$.
If $u$ is a minimizer of $\Fhat$, then also $u_m$ is
a minimizer of $\Fhat$; moreover,
\[
\Ftilde(u_m) = \Fhat(u_m) \leq \Fhat(v_m)
= \Ftilde(v_m) \leq \Ftilde(v),
\qquad
\forall v\in W^{1,1}_0(B_R),
\]
so that $u_m$ is a minimizer of $\Ftilde$.

Hence, we have proved the following

\smallskip
\textsl{Claim 1:
If $u$ is a minimizer of $\Fhat$,
then $u_m$ is a minimizer of both $\Fhat$ and $\Ftilde$.}

\medskip
After this preliminary reduction, let us prove (i)--(iv).

\smallskip
 
(i) Thanks to Claim~1 and Theorem~\ref{t:radial}, assertion~(i)
is a consequence of the following 

\smallskip
\textsl{Claim~2:
There exists a Lipschitz continuous, monotone non-decreasing minimizer $\rad$ of
$\Fhatrad$ in $\Wbase$ satisfying $m \leq \rad \leq 0$.}

Specifically, from (hg) we have that
\[
0\leq \htilde'_-(t) \leq \htilde'_+(t) \leq \htilde'_-(0) =: K < \frac{N\, M}{R},
\qquad \forall t\in\R.
\]
Hence, from Theorem~\ref{t:noncoercive}
the functional $\Fhatrad$ admits a 
Lipschitz continuous
minimizer $\widehat{u}\in\Wbase$.

Let us define
\[
S := \left\{
r \in ]0,R[:\
\widehat{u}_m(r) > \inf_{[r, R]} \widehat{u}_m
\right\}\,.
\]
By Riesz's \textsl{Rising sun Lemma}, we have that
$S$ is the union of a finite or countable family
$(a_k, b_k)$, $k\in J$, of pairwise disjoint open intervals,
with $\hat{u}_m(a_k) = \hat{u}_m(b_k)$ for every $k$
(unless $a_k = 0$, in which case $\widehat{u}_m(0) \leq \widehat{u}_m(b_k)$).
Hence, the function
\[
\rad(r) :=
\begin{cases}
\widehat{u}_m(b_k), & \text{if $r\in (a_k, b_k)$ for some $k\in J$},
\\
\widehat{u}_m(r), & \text{otherwise},
\end{cases}
\]
is a Lipschitz continuous, monotone non-decreasing function and
$\Fhatrad(\rad) \leq 
\Fhatrad(\widehat{u})$,
i.e., $\rad$ is a minimizer of $\Fhatrad$
with the required properties, and Claim~2 is proved.

\medskip
(ii) Here and in the following, $\rad$ will denote the minimizer of $\Fhatrad$
constructed in Claim~2.
By Theorem~\ref{t:noncoercive},
there exists a momentum $p\in\Wbasestar$ such that the Euler--Lagrange
inclusions~\eqref{f:EL1t}--\eqref{f:EL2t} are satisfied 
with $h$ replaced by $\htilde$.
Observing that $m\leq \rad \leq 0$, and
\[
\partial \htilde(0) = \{h'_-(0)\} \subseteq \partial h(0),
\qquad
\partial \htilde(m) \subseteq \partial h(m),
\qquad
\partial \htilde (t) = \partial h(t), \quad\forall t\in ]m,0[,
\]
the same pair satisfies also the Euler--Lagrange
inclusions~\eqref{f:EL1t}--\eqref{f:EL2t} 
(with the original $h$).

\medskip
(iii)
Let us first prove the claim under the additional assumption
that $\htilde\in C^2$.
In this case, the inclusion~\eqref{f:EL1t} is, in fact, the equation
\[
p'(r) = r^{N-1} h'(\rad(r)),
\qquad
\text{for a.e.}\ r\in [0, R],
\] 
$p$ is monotone non-decreasing,
and $p'$ is Lipschitz continuous.

Since $\rad$ is monotone non-decreasing,
there exists $r_0\in [0, R[$ such that
$\rad(r) = m$ for every $r\in [0, r_0[$,
and $\rad(r) > m$ for every $r\in ]r_0, R]$.
Hence, to prove that $x \mapsto \rad(|x|)$ is convex
in $B_R$, it is enough to prove that
$\rad'$ is (equivalent to) a non-decreasing function in $[r_0, R]$.

Moreover,
by \eqref{f:EL2t}, the explicit form \eqref{f:gtsub} of
$\partial\gtilde$, and the monotonicity of $\mu$, 
this property will follow once we prove that
$r^{1-N} p(r)$ is strictly increasing in $]r_0, R]$.

For $r \in ]r_0, R]$, we have that $h'_-(\rad(r)) > 0$,
hence $p'(r) > 0$. 
As a consequence, $p$ is strictly positive and 
strictly monotone increasing in $]r_0, R]$.

Let us fix $\delta\in ]0,1]$.
We have that
\begin{equation}
\label{f:pdelta}
[r^{1-N-\delta} p(r)]'
= r^{-N-\delta}[r\, p'(r) - (N-1+\delta) p(r)]
=: r^{-N-\delta} \lambda(r).
\end{equation}
Since $0\leq p'(r) \leq K\, r^{N-1}$,
the function $\lambda(r) := r\, p'(r) - (N-1+\delta) p(r)$ is
absolutely continuous in $[0,R]$ and $\lambda(0) = 0$.
Moreover, since the function $r\mapsto h'(\rad(r))$ is
monotone non-decreasing,
\[
\begin{split}
\lambda'(r) & = [r^N h'(\rad(r)) - (N-1+\delta) p(r)]'
\\ & = N r^{N-1} h'(\rad(r)) + r^N[h'(\rad(r))]' - (N-1+\delta) p'(r)
\geq (1-\delta) p'(r) \geq 0.
\end{split}
\] 
Hence, $\lambda(r) \geq 0$ for every $r$, so that
from \eqref{f:pdelta} we deduce that the function
$r^{1-N-\delta} p(r)$ is monotone non-decreasing.
As a consequence, the function
$r^{1-N} p(r) = r^\delta [r^{1-N-\delta} p(r)]$
is strictly increasing in $[r_0, R]$.

Finally, the assumption $h\in C^2$ can be dropped as in 
\cite[\S4, Step~3]{C3}
(see also \cite{CM2,CM3}).

\medskip
(iv) 
If $0$ is a strict minimum point of $g$, then
$g$ is strictly monotone increasing in $[0,+\infty[$,
and
the result follows from Theorem~\ref{t:radial}.
If $h$ is a strictly monotone function, the proof can be found
in \cite{CePe1994} (step~(c) in the proof of Theorem~1).
\end{proof}

\begin{example}[The case $N=1$]
Let $N=1$, 
let $\mu\colon [0,R]\to ]0,+\infty]$ be a non-decreasing function,
let $g$ satisfy (g1)--(g2),
and let $h\colon\R\to\R$ be a $C^1$ function
satisfying $0 < h'(t) \leq K < M/R$ for every $t\in\R$.
Then every minimizer $u$ of $F$ in $\W$ is convex.
Specifically, let $\rad(x) = u(|x|)$ and let $p\in\Wbasestar$ be
an associated momentum.
From \eqref{f:EL1} we deduce that
$p'(r) = h'(\rad(r)) > 0$ for every $r\in ]0,R]$,
hence $p$ is a strictly increasing function.
Since $\rad' \geq 0$ and $p(r)\in\partial \gtilde(r, \rad'(r))$,
we conclude that $\rad'$ is non-decreasing,
hence $\rad$ is a convex functions.
\end{example}

\begin{example}
We show that, if $N>1$ and $h$ is not convex, then a minimizer
of $F$ need not be convex.
Let $N=2$, $R=2$, $g(s) = s^2 / 2$, $\mu\equiv +\infty$,
$\varepsilon \in ]0, \sqrt{\log 2}]$,
and consider the function
\[
h(u) :=
\begin{cases}
4(u+\varepsilon), &\text{if}\ u \leq - \varepsilon,\\
0, & \text{if}\ u > \varepsilon.
\end{cases}
\] 
We claim that the non-convex function
\[
\rad(r) :=
\begin{cases}
r^2 - 1 - \varepsilon, & \text{if}\ r\in [0,1],\\
\varepsilon \, \frac{\log(r/2)}{\log 2},
&\text{if}\ r\in [1,2],
\end{cases}
\]
is a minimizer of $\Frad$.
Specifically, the family of all solution of the Euler--Lagrange
inclusions~\eqref{f:EL1}--\eqref{f:EL2} is given by
the trivial pair $(0,0)$ and by the pairs
of the form $(\rad_k, p_k)$, with $k\in\R$,
$p_k(r) = r\, \rad_k'(r)$, and
\[
\rad_k(r) :=
\begin{cases}
r^2 - 1 -\varepsilon + k \, \log r,
&\text{if}\ r\in ]0, 1],
\\
\varepsilon \, \frac{\log(r/2)}{\log 2},
&\text{if}\ r\in [1,2]\,,
\end{cases}
\]
so that $\rad = \rad_0$.
A direct computation shows that
$\Frad(0) = 0$, 
$\Frad(\rad_k) = +\infty$ for every $k\neq 0$,
and $\Frad(\rad) = (\varepsilon^2 - \log 2) / (2 \log 2) < 0$,
hence the claim follows.
\end{example}

\medskip
From the analysis above we can prove the following result without
requiring the convexity of $g$.
In the following, $g^{**}$ denotes the bipolar function of
$z \mapsto g(|z|)$.

\begin{corollary}\label{c:nonconvex}
Let us consider the integral functional \eqref{f:auton},
where $g\colon [0,+\infty[\to [0,+\infty[$
satisfies the following
assumptions:
\begin{itemize}
\item[(g0)]
$g$ is a lower semicontinuous proper function, such that $g(0) = g^{**}(0)$;
\item[(g2)] $M := \liminf_{s\to +\infty} g(s) / s > 0$.
\end{itemize}
Moreover, assume that 
$h\colon\R\to\R$ satisfies (h1) and (hg).
Then $F$ admits a radially symmetric minimizer in $\W$.
\end{corollary}

\begin{proof}
The relaxed functional
\[
\overline{F}(u) := \int_{B_R} \left[
g^{**}(|\nabla u|) + h(u)
\right]\,dx,
\qquad u\in W^{1,1}_0(B_R),
\]
satisfies all the assumptions of Theorem~\ref{t:convex}, hence there exist a radial
minimizer $u(x) = \rad(|x|)$ of $\overline{F}$ in $\W$
and a momentum $p\in\Wbasestar$ such that \eqref{f:EL1t}--\eqref{f:EL2t} hold.

As in the proof of Theorem~\ref{t:convex}(iii), 
considering without loss of generality
$h\in C^2$ and $h'_-(0) > 0$,
we have already proved that $u$ is convex and 
there exists $r_0\in [0, R[$ such that
$\rad(r) = m$ for every $r\in [0, r_0[$,
and $\rad(r) > m$ for every $r\in ]r_0, R]$.
Moreover, the function $r^{1-N} p(r)$ is strictly increasing in $]r_0, R]$.

Let $P$ be the set of all $z\in\R$ such that
$(z, g^{**}(z))$
belongs to the set of the extremal points of the epigraph of $g^{**}$.
We recall that $g(z) = g^{**}(x)$ for every $z\in P$
(see \cite[Remark~5.3]{C4}).
Reasoning as in \cite{CePe1994} (see the proof of Theorem~2),
from the strict monotonicity of $r^{1-N} p(r)$ in $]r_0, R]$
follows that $|\rad'(r)| \in P$ for a.e.\ $r\in [r_0,R]$.
Since $\rad'(r) = 0$ for every $r\in [0,r_0[$,
we conclude that $\overline{F}_{\textrm{rad}}(\rad) = \Frad(\rad)$,
hence $\rad$ is a minimizer of $\Frad$.
\end{proof}

\def\cprime{$'$}

\end{document}